\documentclass{amsart}
\usepackage{cases}
\usepackage{amsmath}
\usepackage{amssymb}
\usepackage{amsfonts}
\usepackage{graphicx}
\newtheorem{theorem}{Theorem}[section]

\newtheorem{lemma}[theorem]{Lemma}

\newtheorem*{problem*}{Problem 1}
\newtheorem{proposition}[theorem]{Proposition}
\newtheorem{remark}[theorem]{Remark}

\newtheorem{thm}{Theorem}

\begin{document}
\title[Rigidity of closed metric measure spaces]
{Rigidity of closed metric measure spaces\\
with nonnegative curvature}
\author{Jia-Yong Wu}
\address{Department of Mathematics, Shanghai Maritime University,
1550 Haigang Avenue, Shanghai 201306, P. R. China}
 \email{jywu81@yahoo.com}
\thanks{This work was partially supported by the NSFC
(11101267, 11271132) and the Innovation Program of
Shanghai Municipal Education Commission (13YZ087).}
\subjclass[2010]{Primary 53C24; Secondary 53C21, 35P15}
\date{Published in Kodai Math. J. 39 (2016) 489-499.
Originally submitted to another journal on Feb. 16, 2013.}
\keywords{Bakry-\'{E}mery Ricci curvature,
weighted Laplacian, eigenvalue, rigidity}
\begin{abstract}
We show that one-dimensional circle is the
only case for closed smooth metric measure spaces with
nonnegative Bakry-\'{E}mery Ricci curvature whose
spectrum of the weighted Laplacian has an optimal
positive upper bound. This result extends the work of
Hang-Wang in the manifold case (Int. Math.
Res. Not. 18 (2007), Art. ID rnm064, 9pp).
\end{abstract}
\maketitle

\section{Introduction}
Let $(M,g)$ be an $n$-dimensional closed Riemannian manifold and
$f\in C^2(M)$. We define a weighted Laplacian on $M$
\[
\Delta_f:=\Delta-\nabla f\cdot\nabla,
\]
which is a self-adjoint operator with respect to the weighted
measure $e^{-f}dv$ (for short $d\mu$),  where $dv$ is the volume
element induced by the metric $g$. The weighted Laplacian is
very much related to the Laplacian of a suitable conformal change
of the background Riemannian metric. It also naturally arises in
potential theory, probability theory and harmonic analysis on complete
Riemannian manifolds. Here, the triple $(M,g,e^{-f}dv)$ is customarily
called a smooth metric measure space.

On the smooth metric measure space $(M,g,e^{-f}dv)$, Bakry-\'{E}mery \cite{[BE]}
(see also \cite{[LD],[Lott1]}) introduced the Bakry-\'{E}mery Ricci curvature
\[
Ric_f:=Ric+Hess(f),
\]
where $Ric$ denotes the Ricci curvature of the manifold and $Hess$
denotes the Hessian with respect to the Riemannian metric.
A remarkable feather of $Ric_f$ is that this tensor relates
to the weighted Laplacian via the following Bochner formula
\begin{equation}\label{Boch}
\Delta_f|\nabla u|^2=2|Hess(u)|^2
+2\langle\nabla u,\nabla\Delta_f u\rangle
+2Ric_f(\nabla u,\nabla u).
\end{equation}
Moreover, Bakry-\'{E}mery Ricci curvature is
related to the gradient Ricci soliton:
\[
Ric_f=\lambda g,
\]
where $\lambda$ is some real constant. The gradient Ricci
soliton is called expanding, steady and shrinking, accordingly
when $\lambda<0$, $\lambda=0$ and $\lambda>0$. As we
all know, The Ricci soliton plays an important role in the theory
of the Ricci flow \cite{[Cao1]}. It is a special solution of the Ricci
flow and often arises from the blow up analysis of the singularities
of the Ricci flow \cite{[Hamilton]}.

By the variational characterization,
the first nontrivial eigenvalue of the weighted Laplacian
on closed metric measure space $(M,g,e^{-f}dv)$ with
respect to the weighted measure $d\mu$ is defined by
\[
\lambda_1:=\inf\limits_{\phi\neq 0}
\left\{\int_M (\nabla\phi,\nabla \phi)d\mu,
:\int_M |\phi|^2d\mu=1, \int_M \phi d\mu=1, \phi\in C^{\infty}(M) \right\}.
\]
The above infimum can be achieved by some smooth
eigenfunction $\phi$. Meanwhile the eigenfunction $\phi$
satisfies the Euler-Lagrange equation
\[
\Delta_f \phi=-\lambda_1 \phi.
\]
We easily see that if potential function $f$ is constant, then $Ric_f$
recovers the ordinary Ricci curvature and the above formulas
all reduce to the classical case.

Many interesting rigid results involving Bakry-\'{E}mery
Ricci curvature have been studied in large part due to their similar
properties between Bakry-\'{E}mery curvature and Ricci curvature.
We refer the readers to \cite{[BE2]}, \cite{[CSW]}, \cite{[FLZ]}, \cite{[HPW]},
\cite{[JaWy]}, \cite{[Kuw]}, \cite{[LD2]}, \cite{[Lott1]}, \cite{[WW]},
\cite{[Wu1]}, \cite{[Wu2]} and reference therein. In particular,
Munteanu and Wang \cite{[MuWa],[MuWa2]}, Su and Zhang \cite{[SuZh]},
and the author \cite{[Wu3]} proved many interesting splitting
results on complete noncompact metric measure spaces under some
assumptions on Bakry-\'{E}mery Ricci curvature. Recently, various
Liouville-type theorems on smooth metric measure spaces were
obtained, see for example \cite{[PRS]} and \cite{[Wu4]}--\cite{[WuWu3]}.

\vspace{0.5em}

In this paper, we continue to discuss a rigid result on the
closed smooth metric measure space rather than the complete
noncompact case. Before introducing our result, we first
recall some well-known eigenvalue estimates on
closed smooth manifolds with nonnegative Ricci curvature.
As we all know, Li and Yau \cite{[Li-Yau]} applied gradient
estimate technique to give a lower bound of the first
eigenvalue of the Laplace operator on a closed manifold
with nonnegative Ricci curvature:
\[
\lambda_1\geq \frac{\pi^2}{2d^2},
\]
where $d$ is the diameter of the manifold. Later, Zhong and Yang
\cite{[Zh-Ya]} improved this result to
\[
\lambda_1\geq \frac{\pi^2}{d^2}.
\]
Recently, there exist some alternate proofs of this result
in \cite{[AC1],[AC2]} and \cite{[Ni]} . We also see that the
above estimate is optimal as equality holds on $S^1$.
Moreover, Hang and Wang \cite{[HangWang]} proved
that $S^1$ is the \emph{only} case for the case
$\lambda_1={\pi^2}/{d^2}$. Their proof relies on a strong
maximum principle and a careful geometrical analysis, which
is not only to simply analyze the proof course of the
Zhang-Yang's inequality becoming the equality. On the
other hand,  Zhong-Yang's result was extended by Chen and
Wang \cite{[ChWa1],[ChWa2]} via probabilistic approach, and
further generalized by Bakry and Qian \cite{[BE3]} to the
smooth metric measure spaces. In particular, they proved that
\begin{thm}\label{T1}
Let $(M,g,e^{-f}dv)$ be a closed smooth metric measure space
with nonnegative Bakry-\'{E}mery Ricci curvature. Then
\[
\lambda_1\geq \frac{\pi^2}{d^2},
\]
where $d$ is the diameter of the manifold $M$.
\end{thm}
We remark that Theorem \ref{T1} has been generalized by
B. Andrews and L. Ni \cite{[AN]}, and A. Futaki and Y. Sano
\cite{[FuSa]}, and further improved by A. Futaki, H.-Z. Li
and X.-D. Li \cite{[FLL]} based on the arguments of Chen and
Wang \cite{[ChWa1],[ChWa2]}.

Motivated by the Hang-Wang's result \cite{[HangWang]}, we
may ask if there exists a Hang-Wang type rigid result
in closed smooth metric measure spaces. That is to say whether
or not $S^1$ is the \emph{only} example for the case
$\lambda_1={\pi^2}/{d^2}$ in Theorem \ref{T1}? The
purpose of this short note is to give an affirmative answer.
Our main result is
\begin{theorem}\label{main}
Let $(M,g,e^{-f}dv)$ be a closed smooth metric measure space
with nonnegative Bakry-\'{E}mery Ricci curvature. Assume that
the first nontrivial eigenvalue of the weighted Laplacian
satisfies
\[
\lambda_1=\frac{\pi^2}{d^2},
\]
where $d$ is the diameter of the manifold $M$. Then $M$ is
isometric to the circle of radius $\frac d\pi$ and $f$ is
constant.
\end{theorem}

The main arguments to prove Theorem 1.1 comes from Hang-Wang
\cite{[HangWang]}, where the gradient estimate, the maximum principle and
some analysis trick are explored. In our case,  the proof not only depends
on Hang-Wang's arguments \cite{[HangWang]}, but also relies on
the weighted gradient estimate and the weighted Bochner formula.
If $f$ is constant, then Theorem \ref{main} returns to
Hang-Wang's result.

\begin{remark}
Recently, S. Lakzian \cite{[Lak]} extended Hang-Wang's rigidity result
to a general setting of metric measure spaces satisfying $RCD(0,N)$
curvature-dimension conditions. If manifold $M$ is complete noncompact,
Munteanu and Wang \cite{[MuWa]} established a sharp upper
bound of the first nonzero eigenvalue of the weighted Laplacian in
terms of the linear growth rate of $f$. They also proved that if
equality holds on the eigenvalue upper estimate and $M$ is not
connected at infinity, then $M$ must be a cylinder.
\end{remark}

By modifying the proof of Theorem \ref{main}, we also
have a similar result for the first nonzero eigenvalue of
the weighted Laplacian with respect to the Neumann boundary
condition of a smooth metric measure space.
\begin{theorem}\label{main2}
Let $(M,g,e^{-f}dv)$ be a compact smooth metric measure space
with nonnegative Bakry-\'{E}mery Ricci curvature and nonempty
convex boundary. Then the first nontrivial eigenvalue of the
weighted Laplacian with respect to the Neumann boundary
condition satisfies
\[
\mu_1\geq\frac{\pi^2}{d^2},
\]
where $d$ is the diameter of the manifold $M$. Moreover if
the above inequality becomes equality, then $M$ is
isometric to a line segment and $f$ is constant.
\end{theorem}

The rest of this paper is organized as follows. In Section
\ref{theor1}, we first recall the proof of Theorem \ref{T1}
and then give an important lemma (see Lemma \ref{Lem2}). In
Section \ref{theor3}, we apply Lemma \ref{Lem2} and the
strong maximum principle to prove Theorem \ref{main}.

\textbf{Acknowledgement}.
The author would like to thank the referee for pointing out mini errors and
making valuable suggestions for the earlier version of this paper.

\section{A key lemma} \label{theor1}
In this section, we first recall the proof of Theorem \ref{T1},
which nearly follows the proofs of Li-Yau \cite{[Li-Yau]} and
Zhong-Yang \cite{[Zh-Ya]}. Here we sketch
the proof for the reader's convenience. Let $(M,g,e^{-f}dv)$ be
a closed smooth metric measure space. Let $\phi$ be the first
eigenfunction of the weighted Laplacian. By multiplying with
a constant it is possible to arrange that
\[
a-1=\min_M \phi,\quad a+1=\max_M \phi,
\]
where $0\leq a(\phi)<1$ is the median of $\phi$. Letting
$u=\phi-a$, then
\[
\Delta_fu=-\lambda_1(u+a).
\]
Following the arguments of \cite{[Li-Yau]} and \cite{[Zh-Ya]}, we can
establish the following gradient estimate of the function $u$.
\begin{proposition}\label{L301}
Let $(M,g,e^{-f}dv)$ be a closed smooth metric measure space
with nonnegative Bakry-\'{E}mery Ricci curvature. Then
\begin{equation}\label{grad}
|\nabla u|^{2}\leq\lambda_1 (1-u^2)+2a\lambda_1
z(u),
\end{equation}
where $u=\phi-a$ and
\[
z(u)=\frac{2}{\pi}\left(\arcsin
u+u\sqrt{1-u^2}\right)-u,\quad u\in [-1,1].
\]
\end{proposition}
It is clear that $z(u)$  is continuous on $[-1,1]$ and $z(-u)=-z(u)$.
From Proposition \ref{L301}, we can deduce $\lambda_1\geq{\pi^2}/{d^2}$ as
follows. Let $x_1,x_2 \in M$, such that $u(x_1)$ is the maximizing
point and $u(x_2)$ is the minimizing point. Take a minimal geodesic
$\gamma$ from $x_2$ to $x_1$ with length at most $d$. Integrating
the estimate \eqref{grad} along this segment with respect
to arclength and using oddness,
\begin{equation*}
\begin{aligned}
d\lambda^{1/2}_1\geq\lambda^{1/2}_1\int_\gamma ds
&\geq\int_\gamma\frac{|\nabla u|ds}{\sqrt{1-u^2+2az(u)}}\\
&\geq\int^1_0\left(\frac{1}{\sqrt{1-u^2+2az}}+\frac{1}{\sqrt{1-u^2-2az}}\right)du\\
&\geq\int^1_0\frac{1}{\sqrt{1-u^2}}\left(2+\frac{3a^2z^2}{1-u^2}\right)du\\
&\geq\pi+3a^2\left(\int^1_0\frac{zdu}{\sqrt{1-u^2}}\right)^2\\
&=\pi+\frac{3a^2}{\pi^2}\left(\frac{\pi}{2}-1\right)^4.
\end{aligned}
\end{equation*}
Hence $\lambda_1\geq{\pi^2}/{d^2}$
and the inequality is strict unless $a=0$ (i.e. $\min_M u=-1$).

From the above proof, we easily see that on $S^1$, the above inequalities all
become equality. Naturally, we ask if $S^1$ is the only case for the
equality case. The answer is YES! In the rest of this note, we will explain this fact.

At first, we derive a differential inequality on the dense open set which
consists of all regular points of the eigenfunction.
\begin{lemma}\label{Lem2}
Let $(M^n,g,e^{-f}dv)$ be a closed smooth metric measure space. Let
$u$ be a nonzero smooth function on this measure space such that
\[
\Delta_fu=-\lambda u.
\]
Then on $\Omega=\{\nabla u\neq 0\}$,
\begin{equation}\label{inequ}
\Delta_f\psi-\frac{\nabla \psi\cdot\nabla(\psi-2\lambda u^2)}{2|\nabla u|^2}
\geq2 Ric_f(\nabla u,\nabla u),
\end{equation}
where $\psi:=|\nabla u|^2+\lambda u^2$.
\end{lemma}
\begin{proof}[Proof of Lemma \ref{Lem2}]
The proof of this result follows from that of Lemma 1 in
\cite{[HangWang]} with little modification, but is included
for completeness. Following the computation method of \cite{[Li-Yau]},
let $e_1, e_2,..., e_n$ be a local orthonormal frame field on $M^n$.
We adopt the notation that subscripts in $i$, $j$, and $k$, with $1\leq i, j, k\leq n$,
mean covariant differentiations in the $e_i$, $e_j$ and $e_k$, directions respectively.

Differentiating $\psi$ in the direction of $e_i$, we have
\[
\psi_i=2\sum_ju_ju_{ij}+2\lambda uu_i,
\]
and so
\[
\left|\frac 12\nabla\psi-\lambda u\nabla u\right|^2
=\sum_i\left(\sum_ju_ju_{ij}\right)^2
\leq|\nabla^2u|^2\cdot|\nabla u|^2,
\]
where the summation convention is adopted on repeated indices.
This implies
\[
\frac 14|\nabla\psi|^2-\lambda u\nabla u\cdot\nabla\psi
\leq|\nabla u|^2(|\nabla^2u|^2-\lambda^2u^2).
\]
Therefore on $\Omega=\{\nabla u\neq 0\}$, we have
\begin{equation}\label{guanxi}
|\nabla^2u|^2-\lambda^2u^2\geq\frac{|\nabla\psi|^2
-4\lambda u\nabla u\cdot\nabla\psi}{4|\nabla u|^2}
=\frac{\nabla(\psi-2\lambda u^2)\cdot\nabla\psi}{4|\nabla u|^2}.
\end{equation}
On the other hand, using the Bochner formula \eqref{Boch},
we conclude that
\begin{equation*}
\begin{aligned}
\frac 12\Delta_f\psi&=|\nabla^2u|^2+\nabla u\cdot\nabla\Delta_fu
+Ric_f(\nabla u,\nabla u)+\lambda|\nabla u|^2+\lambda u\Delta_fu\\
&=|\nabla^2u|^2-\lambda^2u^2+Ric_f(\nabla u,\nabla u),
\end{aligned}
\end{equation*}
where we used $\Delta_fu=-\lambda u$. Combining this
with \eqref{guanxi} yields \eqref{inequ}.
\end{proof}

\section{Proof of Theorem \ref{main}} \label{theor3}
In this section we will prove Theorem \ref{main}. Since the idea of
proof comes from Hang-Wang \cite{[HangWang]}, we only provide main
steps and omit tedious discussions.

\begin{proof}[Proof of Theorem \ref{main}]
Assume that $\lambda_1=\pi^2/d^2$. From the proof of Proposition
\ref{L301} in Section \ref{theor1}, we easily get
$a=0$, and hence
\[
\min_M u=-1 \quad \mathrm{and} \quad \max_M u=1,
\]
where $u=\phi$ is a first eigenfunction. By scaling the metric,
we can assume $d=\pi$. So $\lambda_1=1$. Let $\psi=|\nabla u|^2+u^2$.
By Lemma \ref{Lem2}, on $\Omega=\{\nabla u\neq 0\}$, we have
\begin{equation}\label{inequR0}
\Delta_f\psi-\frac{\nabla \psi\cdot\nabla(\psi-2\lambda u^2)}{2|\nabla u|^2}
\geq 0.
\end{equation}
Using the maximum principle, we conclude that
\[
\psi:=|\nabla u|^2+u^2\leq \max_{\{\nabla u=0\}}(|\nabla u|^2+u^2)=1,
\]
since $\psi$ can not attain the maximum value at the point of
$\Omega=\{\nabla u\neq 0\}$.

Take two points $p_0$ and $p_1$
such that
\[
u(p_0)=-1 \quad \mathrm{and} \quad u(p_1)=1,
\]
and let $\gamma:[0,l]\to M$
be a unit speed minimizing geodesic from $p_0$ to $p_1$.
We define a function $y(t)=u(\gamma(t))$. Then
\[
|y'(t)|=|\nabla u(\gamma(t))\cdot\gamma'(t)|
\leq |\nabla u(\gamma(t))|\leq\sqrt{1-y^2(t)}.
\]
Hence
\[
\pi\geq l\geq\int_{\{0\leq t\leq l,y'(t)>0\}}dt
\geq\int^l_0\frac{y'(t)dt}{\sqrt{1-y^2(t)}}
=\int^1_{-1}\frac{dx}{\sqrt{1-x^2}}=\pi.
\]
Therefore
\[
l=\pi\quad \mathrm{and}\quad y'(t)>0
\]
for almost every $t\in(0,\pi)$. Hence $y(t)$ is
strictly increasing on $[0,\pi]$.
Moreover, we also have
\[
\int^{\pi}_0\frac{y'(t)dt}{\sqrt{1-y^2(t)}}=\pi,
\]
which implies that $y'(t)=\sqrt{1-y^2(t)}$ for all
$t\in [0,\pi]$. Since $y(0)=-1$ and $y'(0)=0$, then
\[
y(t)=u(\gamma(t))=-\cos t
\]
for $t\in [0,\pi]$. It follows that
\[
(\nabla^2u)(\gamma'(0),\gamma'(0))=1.
\]
Since $\Delta_fu(p_0)=-\lambda_1u(\gamma(0))=1$,
$(\nabla f\cdot\nabla u)(p_0)=0$
and $(\nabla^2u)_{p_0}\geq 0$, we conclude that
$\Delta u(p_0)=1$ and hence we must have
\[
(\nabla^2u)_{p_0}=\lambda_{\gamma'(0)}\otimes\lambda_{\gamma'(0)},
\]
where for any tangent vector $X$, $\lambda_X$ is the dual cotangent
vector given by $\lambda_X(Y)=\langle X,Y\rangle$ for any tangent
vector $Y$.

\vspace{0.5em}

Next, similar to the Hang-Wang's argument \cite{[HangWang]},
we get
\begin{proposition}
The set $\{u=\pm 1\}$ has at most four points.
\end{proposition}
\begin{proof}
We only discuss the case $\{u=1\}$ since the case $\{u=-1\}$ is
similar. For any point $p$ with $u(p)=1$, we choose a minimizing
geodesic $\gamma_p:[0,l_p]\to M$ from $p_0$ to $p$. Then the
same argument as before shows that
\[
l_p=\pi \quad \mathrm{and} \quad
(\nabla^2u)_{p_0}=\lambda_{\gamma_p'(0)}\otimes\lambda_{\gamma_p'(0)},
\]
which implies $\gamma_p'(0)=\pm\gamma'(0)$. Hence $p=\exp(\pi\gamma_p'(0))$
has at most two choices.
\end{proof}

In the next step, to finish the proof of Theorem \ref{main}, we only
need to claim that the dimension of $M$ must be one. Argue by contradiction.
If the claim is not true, then we assume that
$\mathrm{dim} M\geq2$. If we let
\[
M^*=M\setminus\{u=\pm1\},
\]
then $M^*$ is still connected. In the following we want to show
$|\nabla u|^2+u^2=1$ on $M^*$. In fact, we consider
\[
E=\{p\in M^*:|\nabla u(p)|^2+u^2(p)=1\}.
\]
Clearly, $E$ is closed. On the other hand, if $p\in E\subset \Omega$,
by \eqref{inequR0} and the strong maximum principle, we have
\[
|\nabla u|^2+u^2\equiv1
\]
near $p$. Hence $E$ must be either an empty set or $M^*$.
Since for any $t\in(0,\pi)$,
\[
|\nabla u(\gamma(t))|^2+u^2(\gamma(t))\geq \cos^2t+\sin^2t
=1,
\]
we see $E$ is nonempty and therefore $E=M^*$. Now we define
$X=\frac{\nabla u}{|\nabla u|}$ on $M^*$. Since
$|\nabla u|^2+u^2\equiv1$, differentiating it yields
\[
\nabla^2u(X,X)=-u.
\]
We also notice that the proof of Lemma \ref{Lem2}
easily implies that
\[
|\nabla^2u|^2=u^2
\]
on $M^*$, since $\psi=|\nabla u|^2+u^2\equiv1$. Combining the
above two equalities, we have
\[
\nabla^2u=-u\lambda_X\otimes\lambda_X.
\]
Direct calculation shows that $\nabla_XX=0$, and hence all
integral curves of $X$ are geodesics. Let $\Sigma=\{u=0\}$.
Since $|\nabla u|=1$ on $\Sigma$, we see that $\Sigma$ is a
hypersurface, which may have more than one components.
For any $p\in\Sigma$, let $\alpha_p$ be the maximal integral
curve of $-X$ with $\alpha_p(0)=p$. Then $\alpha_p$ is a unit
speed geodesic. Letting $y_p(t)=u(\alpha_p(t))$, we know that
\[
y_p(0)=0\quad \mathrm{and} \quad y_p'(t)=-\sqrt{1-y^2_p(t)}.
\]
It gives that
\[
y_p(t)=-\sin t \quad \mathrm{for} \quad t\in [0,\pi/2).
\]
On the other hand, $\alpha_p$ is a geodesic on $M$,
defined on $[0,\infty)$. We have
\[
u(\alpha_p(t))=-\sin t
\]
for $t\in [0,\pi/2]$. In particular, $u(\alpha_p(\pi/2))=-1$.
The same argument as before shows
\[
(\nabla^2u)_{\alpha_p(\frac{\pi}{2})}=\lambda_{\alpha'_p(\frac{\pi}{2})}
\otimes\lambda_{\alpha'_p(\frac{\pi}{2})}.
\]
Here $p=\exp_{\alpha_p(\frac{\pi}{2})}
\left(-\frac{\pi}{2}\alpha'_p(\frac{\pi}{2})\right)$.
Since there are at most two points in the set $\{u=-1\}$, we may
find point $q$ satisfying $u(q)=-1$ and infinitely many $p\in\Sigma$
such that $\alpha_p(\frac{\pi}{2})=q$. This clearly leads to a
contradiction since  $\alpha'_p(\frac{\pi}{2})$ has at most two
choices. Therefore the dimension of $M$ must be one. At this time,
we easily see that $Ric(M)=0$ and $Hess(f)\geq 0$ on $S^1$. Hence
$f''(t)=0$ on $S^1$, and $f$ is constant.
\end{proof}

\bibliographystyle{amsplain}

\end{document}